\documentclass[12pt]{article} 

\usepackage{amsmath}
\usepackage{amssymb}
\usepackage{amsthm}
\usepackage[all,cmtip]{xy}
\usepackage{fancyhdr}
\usepackage{euscript}
\usepackage{setspace}
\usepackage{graphicx,epsfig}
\usepackage{palatino, url, multicol}
\usepackage{mathrsfs} 
\usepackage{enumerate}   
\usepackage{verbatim} 

\theoremstyle{definition}

\newtheorem{theorem}{Theorem}[section]

\newtheorem{proposition}[theorem]{Proposition}

\newtheorem{corollary}[theorem]{Corollary}

\newcommand{\IR}{\hbox{$\mathbb{R}$}}

\newcommand{\IZ}{\hbox{$\mathbb{Z}$}}

\newcommand{\abs}[1]{\hbox{$\left| {#1} \right|$}}

\def\it{\itshape}

\def\tt{\texttt}
\def\bf{\textbf}

\def\IR{\mathbb{R}}

\def\I1{\mathbb{1}}

\def\limx0{\lim_{x \to 0}}

\def\intxyleq1{\underset{\| x - y  \| \leq 1}{\int}}
\def\intxygeq1{\underset{\| x - y  \| \geq 1}{\int}}
\def\intxizetaleq1{\underset{\| \xi - \zeta  \| \leq 1}{\int}}
\def\intxizetageq1{\underset{\| \xi - \zeta \| \geq 1}{\int}}

\def\tab{\hskip 1mm}

\def\tab{\hspace{.1pc}}
\def\ttab{\hspace{1pc}}
\newcounter{hours}
\newcounter{minutes}
\newcommand\printtime{%
  \setcounter{hours}{\the\time/60}%
  \setcounter{minutes}{\the\time-\value{hours}*60}%
  \ifthenelse{\value{hours} > 12}
     {
       \setcounter{hours}{\value{hours}-12}%
       \thehours:\theminutes \ p.m.                
     }
     {
       \thehours:\theminutes \ a.m.                
     } 
}

\def\putdate{{\tt Compiled on \the\month-\the\day-\the\year \ at\printtime} \\}

\begin{document} 

\title{Approximation pathologies for certain continued fractions}
 \author{Avraham Bourla\\
   Department of Mathematics\\
  American University, Washington DC\\
   \texttt{bourla@american.edu}}
 \date{\today}
 \maketitle
\begin{abstract}
\noindent We will provide a family of continued fractions for which there is no correspondence between dynamic and approximation pairs, leading to an anomaly in their corresponding  Spaces of Jager Pairs.
\end{abstract}

\section{Introduction and preliminaries}{} 

Associated with every continued fraction expansion is its Space of Jager Pairs, used in assessing the approximation quality for its convergents \cite{JK}. In this paper, we will examine the one parameter family of continued fractions
\[ [a_1,a_2,...]_k := \frac{k}{k+a_1+\frac{k}{k+a_2+ ...}}, \ttab a_n \in \IZ_{\ge 0}, \ttab 0 < k \in \IR,\] 
introduced in \cite{HM1}, was also studied extensively in \cite{Avi1, Avi2}. Letting $x_0 :=[a_1,a_2,...]_k$ leads to the definitions of the convergents $\frac{p_0}{q_0} = \frac{0}{1}$ and 
\[ \frac{p_n}{q_n} := [a_1,a_2,...,a_n]_k = \frac{k}{k+a_1+\frac{k}{k+a_2+ ... + \frac{k}{k+a_n}}}, \tab n\ge1.\]
Define the future and past of $x_0$ at time $n \ge 1$ to be $x_n := [a_{n+1},a_{n+2},...]_k \in (0,1)$ and $y_n := -k-a_n - [a_{n-1},a_{n-2},...,a_1]_k \in (-\infty, -k]$ (we take $y_1 = -k-a_1$) and call the pairs $(x_n,y_n)$ the dynamic pairs of $x_0$ at time $n$. The pair of consecutive terms $(\theta_{n-1}(x_0),\theta_n(x_0))$ in the sequence of approximation coefficients $\{\theta_n\}_0^\infty := \left\{\abs{x_0 - \frac{p_n}{q_n}}q_n^2\right\}_0^\infty$ is called the approximation pair of $x_0$ at time $n$. There is a correspondence between these pairs, which was established by Haas and Molnar in \cite[Theorem 4]{HM1}, who proved that $(\theta_{n-1},\theta_n)$ is the image of $(x_n,y_n)$ under the bijective map 
\begin{equation}\label{Psi}
\Psi_k(x,y) := \left(\frac{1}{x-y}, -\frac{x{y}}{k(x-y)}\right). 
\end{equation}
We will prove that there is no such bijection when $0<k<1$ and, consequently, that the formula for the Space of Jager Pairs for these continued fraction expansions does not follow suit with the $k\ge 1$ cases.

\section{The space of Jager Pairs when k is less than one}{}

\noindent For all $k \in (0,\infty)$ and $a \in \IZ_{\ge 0}$ define the region $P_{(k,a)} := (0,1) \times (-k-a-1, -k-a]$ and $P_{(k,a)}^\# := \Psi_k(P_{(k,a)})$. The continuity of the map $\Psi$ provides us with the partition $\Gamma_k = \Psi_k(\bigcup{P_a}) = \bigcup{\Psi(P_a)} = \bigcup{P_a^\#}$. To find $P_a^\#$, we will use the following propositions:

\begin{proposition}{\cite[Proposition 3.1]{Avi1}}\label{P_a}\\
If $\Psi$ is injective on $P_{(k,a)}$, then $P_{(k,a)}^\#$ is the quadrangle with vertices $\left(\frac{1}{k+a},0\right), \tab \left(\frac{1}{k+a+1},\frac{k+a}{k(k+a+1)}\right),\\ \tab \left(\frac{1}{k+a+2},\frac{k+a+1}{k(k+a+2)}\right)$ and $\left(\frac{1}{k+a+2},0\right)$. 
\end{proposition}

\begin{proposition}\label{Psi_fold}
$\Psi$ is injective on the region $\{(x,y) \in \IR^2 : x + y < 0\}$ and is invariant under the reflection about the line $x+y=0$.% In particular, $\Psi$ is injective on $\Omega'$ if and only if $k \ge 1$.   
\end{proposition}
\begin{proof}
The equality $\Psi(x, y) = \big(\frac{1}{x-y}, -\frac{xy}{k(x-y)}\big) = \Psi(-y,-x)$ proves the first statement. Let $(x_1,y_1)$ and $(x_2,y_2)$ be points which are on or below the line $x + y = 0$, so that $x_1 + y_1 \le 0, \tab x_2 + y_2 \le 0$ and let $u_1,v_1,u_2,v_2 \in \IR$ be such that $(u_1,v_1) = \Psi(x_1,y_1)=  \Psi(x_2,y_2) = (u_2,v_2)$. Then the definition of $\Psi$ \eqref{Psi} implies that $\frac{1}{x_1-y_1} = u_1 = u_2 = \frac{1}{x_2-y_2}$, hence
\begin{equation}\label{x-y_G}
x_1 - y_1 = x_2 - y_2 \ne 0.
\end{equation}
Also 
\[-\dfrac{u_1}{k}x_1y_1 = v_1 = v_2 = -\dfrac{u_2}{k}x_2y_2  =  -\dfrac{u_1}{k}x_2y_2\]
and $u_1,u_2 > 0$ imply that  $x_1y_1 = x_2y_2$, so that
\[(x_1 +y_1)^2 = (x_1 - y_1)^2 + 4x_1y_1 = (x_2 - y_2)^2 + 4x_2y_2 = (x_2 + y_2)^2.\]
Since $x_1 + y_1 \le 0$ and $x_2 + y_2 \le 0$, this last equation proves $x_1 + y_1 = x_2 + y_2$, which in tandem with condition \eqref{x-y_G}, proves $x_1=x_2$ and $y_1 = y_2$, hence $\Psi$ is injective on or below the line $x+y=0$. %When $k \ge 1$ this condition holds for all $(x,y) \in \Omega$ and implies that $\Psi$ is injective on $\Omega'$.  When $0<k<1, \tab k<x_0<1$ and $-1<y_0<-k$, both the points $(x_0, y_0)$ and $(-y_0, -x_0)$, belong to $\Omega'$. Since their image under $\Psi$ is identical, $\Psi$ is not injective on $\Omega'$ in this case. 
\end{proof}

\noindent These propositions proves that when $k \ge 1$, the Space of Jager Pairs 
\[\Gamma_k := \{(\theta_{n-1}(x_0),\theta_n(x_0))\}_{x_0 \in (0,1)}\] 
is the convex quadrangle with vertices $\left(0,0\right), \tab \left(\frac{1}{k},0\right),\tab\left(\frac{1}{k+1},\frac{1}{k+1}\right)$ and $\left(0,\frac{1}{k+1}\right)$.

\begin{center}
\includegraphics[scale=.5]{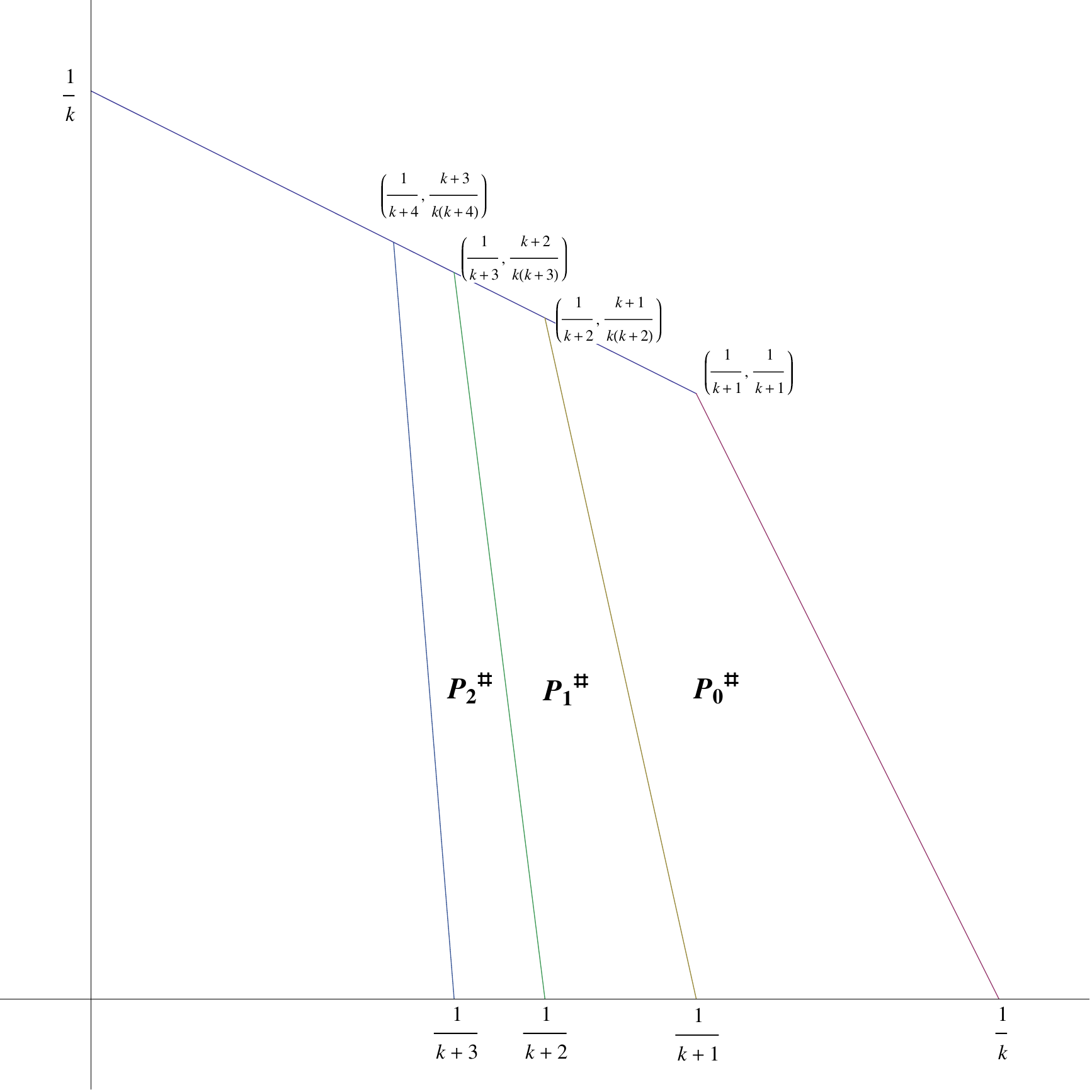}\\
$\Gamma_k \text{ when $k\ge 1$}$\\
\end{center}

\noindent When $0<k<1$, $\Psi$ is injective on $P_{(k,a)}$ precisely when $a>0$. In order to finish the characterization of the space of approximation coefficients for these cases, we prove:
\begin{theorem}\label{P_0_k<1}
When $0<k<1$, $P_0^\#$ is the intersection of the unbounded regions $u{k}+v < 1, \tab v > 0, \tab  (k+1)^2u+k{v} > k+1, \tab u+k{v} < 1$ and $4k{u}{v} \le 1$. 
\end{theorem}
\begin{proof}
Let $(u,v) := \Psi(x,-k-a)$. From the definition \eqref{Psi} of $\Psi$, we obtain
\begin{equation}\label{u=f(x,y)}
u = \dfrac{1}{x-y} = \dfrac{1}{x+k+a}
\end{equation}
and $v = -\frac{x{y}}{x-y}$ so that
\begin{equation}\label{v=f(u)}
v =  -\dfrac{x{y}}{k(x-y)} = \dfrac{u}{k}x(-y) = \dfrac{u}{k}\bigg(\dfrac{1}{u}- (a+k)\bigg)(a+k) = \dfrac{a+k}{k}\big(1-(a+k)u\big).
\end{equation}
After restricting $P_0$ to its part on or below $x+y=0$, we are left with the region whose boundary is the 5-gon with vertices $(k,-k), (1,-1), (1,-k), (1,-k-1),  (0,-k-1)$ and $(0,-k)$, which includes the part of its perimeter, which is the open line segment connecting the first vertex to the second vertex. We will use these formulas to evaluate the image of each open line segment in this 5-gon under $\Psi$.
\begin{enumerate}
\item When $y=-k$, we obtain that $u$ ranges between $\frac{1}{k}$ and $\frac{1}{2k}$ as $x$ ranges between $0$ and $k$. Plugging $a=0$ to formula \eqref{v=f(u)} yields that the line segment $(0,k) \times \{-k\}$ in the $x{y}$ plane maps to the open segment in the line $uk+v=1$ between the points $\left(\frac{1}{2k},\frac{1}{2}\right)$ and $\left(\frac{1}{k},0\right)$ on the $u{v}$ plane\\

\item  Plugging $x=a=0$ yields that the line segment $\{0\} \times (-k-1,-k)$  in the $x{y}$ plane maps to $\left(\frac{1}{k+1},\frac{1}{k}\right)\times\{0\}$ on the $u{v}$ plane. 

\item The line segment $(0,1) \times \{-(k+1)\}$  in the $x{y}$ plane maps to the open segment of the line $(k+1)^2u+k{v}=k+1$ between $\left(\frac{1}{k+1},0\right)$ and $\left(\frac{1}{k+2},\frac{k+1}{k(k+2)}\right)$ on the $u{v}$ plane. 

\item Plugging $x=1$ and $a=0$ yields that the line segment $\{1\} \times \left(-1,-(k+1)\right)$ in the $x{y}$ plane maps to the open segment of the line $u+k{v}=1$ between the points $\left(\frac{1}{k+2},\frac{k+1}{k(k+2)}\right)$ and $\left(\frac{1}{2},\frac{1}{2k}\right)$ on the $u{v}$ plane.

\item Finally, if $y=-x$ then $u = \frac{1}{x-y}= \frac{1}{2x}$ hence $x = \frac{1}{2u}$ and $v= -\frac{u}{k}x{y} = \frac{1}{4k{u}}$. Thus the line $x+y=0$  in the $x{y}$ plane is mapped under $\Psi$ to the hyperbola $4k{u}v=1$ in the $u{v}$ plane. The points $(1,-1)$ and $(k,-k)$ map to $\big(\frac{1}{2},\frac{1}{2k}\big)$ and $\big(\frac{1}{2k}, \frac{1}{2}\big)$ under $\Psi$, so that the open segment of the line $x+y=0$ from $(1,-1)$ to $(k,-k)$ maps to the open arc on this hyperbola from $\big(\frac{1}{2},\frac{1}{2k}\big)$ to $\big(\frac{1}{2k}, \frac{1}{2}\big)$. 
\end{enumerate}

\noindent Since $\Psi$ is a continuous bijection on and below the line $x+y=0$, it maps the interior and boundary of this 5-gon bijectively into the interior and boundary of the region in the $u{v}$ plane whose boundary we have just determined and which coincides with the hypothesis. Using proposition \ref{Psi_fold}, we know that the part of $P_0$ which is above the line $x+y=0$ has the same image under $\Psi$ as its reflection about this line. Since this reflection is also contained in $P_0$, we conclude that $P_0$ is mapped in its entirety onto this region, thus concluding the result.      
\end{proof}

\begin{corollary}
When $0<k<1$, the Space of Jager Pairs is the region in the $u{v}$ plane which is the union of the quadrangle vertices $\left(0,0\right), \tab \left(\frac{1}{k},0\right),\tab\left(\frac{1}{k+1},\frac{1}{k+1}\right)$ and $\left(0,\frac{1}{k+1}\right)$ and the part of the hyperbola $4k{u}{v}=1$ in the $u{v}$ plane between $u=\frac{1}{2}$ and $u=\frac{1}{2k}$.  
\end{corollary}

\begin{center}
\includegraphics[scale=.5]{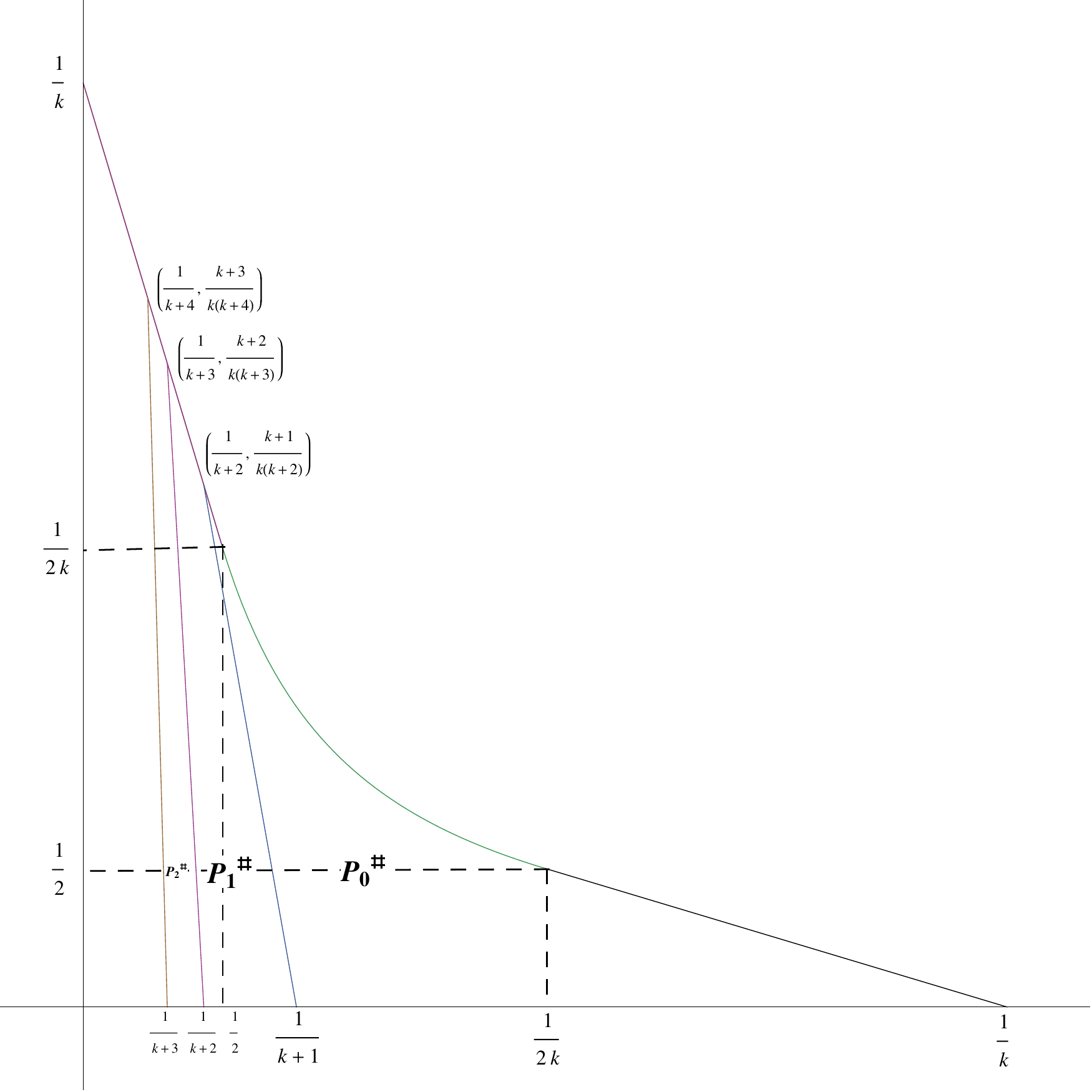}\\
$\Gamma_k \text{ when $0 < k < 1$}$
\end{center}

\end{document}